\documentclass[12pt,twoside]{amsart}
\usepackage{amssymb}
\usepackage{a4wide}
\usepackage[latin1]{inputenc}
\usepackage[T1]{fontenc}
\usepackage{times}

\def\hpq0{h^{p,q}_{\leq 0}}
\def\Hpq0{\H_{\leq 0}^{p,q}}

\def\dbar{\bar\partial}
\def\ddbar{\partial\dbar}

\def\H{{\mathcal H}}

\def\Re{{\rm Re\,  }}

\def\be{\begin{equation}}
\def\ee{\end{equation}}

\newtheorem{thm}{Theorem}[section]
\newtheorem{lma}[thm]{Lemma}

\newtheorem{prop}[thm]{Proposition}

\theoremstyle{definition}

\theoremstyle{remark}

\newtheorem{preremark}{Remark}
\newtheorem{preex}{Example}

\numberwithin{equation}{section}

\begin{document}

\title[]
{ $L^2$-extension of $\dbar$-closed forms}

\address{B Berndtsson :Department of Mathematics\\Chalmers University
  of Technology \\
  and Department of Mathematics\\University of G\"oteborg\\S-412 96
  G\"OTEBORG\\SWEDEN,\\} 

\email{ bob@math.chalmers.se}

\author[]{ Bo Berndtsson}

\begin{abstract}
Generalizing and strengthening a recent result of Koziarz, we prove a
version of the Ohsawa-Takegoshi-Manivel theorem for $\dbar$-closed
forms. 
\end{abstract}

\bigskip

\maketitle

\section{Introduction}

The celebrated Ohsawa-Takegoshi-Manivel extension theorem,
\cite{Ohsawa},\cite{Manivel}  gives optimal
conditions for the extension of holomorphic sections of line bundles
from a divisor to the ambient space. In Manivel's article,
\cite{Manivel}, it is stated that a completely parallell result holds
for smooth $\dbar$-closed forms of higher degree. There is however a problem
in the proof of this in \cite{Manivel} which is connected with the
regularity of solutions of certain $\dbar$-equations with singular
weights. This problem is also discussed in \cite{Demailly}, where a
strategy towards its solution is put forward. 

Recently, an at least moral solution of this problem was given by
Koziarz, \cite{Koziarz}. Instead of looking at the extension of
individual forms, Koziarz considered the extension of cohomology
classes, i e extended closed forms up to a $\dbar$-exact error. This
formulation is actually more natural than the original problem since
cohomology classes have well defined restrictions on divisors, whereas
$\dbar$-closed forms restrict only if a somewhat artificial condition
of smoothness is imposed. Koziarz's method is inspired by work of Siu,
\cite{Siu}, and consists in representing cohomology classes by Cech
cocycles. These cocycles consist of holomorphic objects
for which the available machinery works better.

The purpose of this note is twofold. First we will prove a simple
proposition saying that a smooth $\dbar$-exact form on a divisor can
always be extended to a closed form with arbitrary small $L^2$-norm in
the ambient space. (This property characterizes exact
forms.) This means that Koziarz's theorem on the extension up to an
exact error actually gives a solution to the original problem on
extension of closed forms. Second, we will give an alternative proof
of Koziarz's theorem, following the method in \cite{Berndtsson}. The
advantage with 
this alternative proof is that it gives an absolute constant for the
extension, whereas in Koziarz's theorem the constant depended on the
manifold and the divisor. Moreover, the curvature conditions that
guarantee extendability are shown to be somewhat more liberal for
forms of higher 
degree than for holomorphic sections. Finally, the proof exhibits the
significance of extension of cohomology classes in a seemingly
interesting way.

Let us comment a little bit more on this. If $u$ is a holomorphic
section of $K_\Delta +L$ over a divisor $\Delta$, the method in
\cite{Berndtsson}, see also \cite{Amar}, consists in solving the
equation
$$
\dbar v= u\wedge [\Delta]:=g.
$$
The right hand side here is not a $L^2$-form but a current, but
nevertheless it turns out that $L^2$-methods can be used here. One
cannot however get a solution $v$ in $L^2$. If the divisor $\Delta$ is
defined by a section $s$ of some line bundle $S$ over the ambient manifold
$X$, the solution of the extension problem is $sv$, so what we want is
an $L^2$-estimate for $sv$. Dually, (and formally!) this corresponds
to an estimate for smooth testforms $\alpha$ like
\be
|\langle g,\alpha\rangle|^2\leq C \int |\dbar^*\alpha|^2/|s|_\psi^2
\ee
(where $\psi$ is some metric on $S$). But this dual formulation is
only formal. The fact that the weight $|s|^{-2}$ is nonintegrable
causes a problem in the functional analysis involved since all smooth
test forms do not have finite norm with respect to this weight. This
problem can be circumvented if we instead prove a stronger estimate
\be
|\langle g,\alpha\rangle|^2\leq C \int |\dbar^*\alpha|^2/|s|_\psi^r
\ee
where $r<2$. Then the functional analytic difficulty disappears and
one even gets a stronger result than is asked for. 

We now want to follow the same route for forms of higher degree. Booth
estimates (1.1) and (1.2) can then be proved in much the same manner
as for holomorphic sections. As in the case of holomorphic sections,
the best thing would  be to use (1.2), since that is a bona fide
dual formulation of the $\dbar$-problem. But this causes problems with
regularity. One would then need to dicuss regularity properties in
$L^2$-spaces with singular weights, which leads back to the original
problem with Manivel's argument. We therefore choose to work with
(1.1) instead. Then the regularity problems disappear since we can go
back and forth between estimates with the singular weight $|s|^{-2}$
and estimates without that weight by multiplying and dividing with
$s$. The price we have to pay for this is that (1.1) is no longer a
dual formulation of the $\dbar$-estimate, and so not a dual
formulation of the extension problem. But, miraculously, it turns out
to be a dual formulation of the extension of cohomology classes, and
this is what makes the scheme work.

In this paper we will suppose all the time that $X$ is a compact
K\"ahler manifold. Maybe the same arguments could be pushed to non
compact situations, but the compactness assumption simplifies and
makes the argument a little bit simpler than in \cite{Berndtsson}

\section{$\dbar$-exact forms}
In this section we discuss the extension of $\dbar$-exact forms. The
main point is the following proposition.
\begin{prop}
Let $X$ be an $n$-dimensional compact complex manifold, and let
$\Delta$ be a smooth divisor in $X$. Let $L$ be a holomorphic line
bundle over $X$. Let $u$ be a smooth
$\dbar$-closed $L$-valued $(0,q)$-form on $\Delta$, $q\geq 1$. Then
$u$ is $\dbar$-exact on $\Delta$ if and only if, for any $\epsilon>0$,
there is an extension, $U$, of $u$ to $X$ with $L^2$-norm smaller than
$\epsilon$. 
\end{prop}
Here $L^2$-norms are taken with respect to some smooth metric 
and some arbitrary smooth volume form. In the proof we use 
the next lemma.

\begin{lma}
 There is a sequence of cutoff-functions $\rho_\epsilon$ such
that

1. The sets where $\rho_\epsilon(z)=1$ are neighbourhoods of  $\Delta$
shrinking to $\Delta$, and the sets where $\rho_\epsilon (z)=0$ increase to
$X\setminus \Delta$.

2. $\|\dbar\rho_\epsilon\|$ goes to zero with $\epsilon$.
\end{lma}
\begin{proof} Let first the dimension be 1 and take $X$ to be the unit
  disk and $\Delta$ to 
  be the origin. The main point is that 
there is a complete Kahler metric on the punctured disk, $\omega$,
which gives $\{|z|<1/2\}$ finite area.
Indeed, the Poincare metric
$$
\omega=idz\wedge d\bar z/(|z|^2(\log|z|)^2)
$$
has this property. Completeness implies that there is some realvalued
function near the origin, $\rho$ , such that $\rho(z)$ tends to
infinity when $z$ tends to zero and 
$$
i\partial \rho\wedge\dbar\rho\leq \omega.
$$
Explicitly, $\rho(z)=\log\log(1/|z|)$ will do. Define  functions
$\chi_k(x)$ on the positive halfaxis, equal to 0 when $x<k$ , to 1
when $x>k+1$, and having $\chi_k'$ bounded. Then put 
$$
\rho_\epsilon=\chi_{1/\epsilon}\circ\rho.
$$
Then 1 is clear and 2 follows by dominated convergence since
$$
\int_{|z|<1/2} i\partial\rho_\epsilon\wedge\dbar\rho_\epsilon\leq
\int_{|z|<1/2}\chi_\epsilon'\omega.
$$
The general case is basically the same. We can cover $\Delta$ by a finite
number of coordinate neighbourhoods, inside which $\Delta$ is defined by
the equation $z_1=0$. Then take $\rho_\epsilon(z_1)$ with
$\rho_\epsilon$ defined as above and piece together with a partition
of unity. 
\end{proof}

With this we can turn to the proof of the proposition. Assume first
that $u=\dbar v$ on $\Delta$ with $v$ smooth. We extend $v$ to $X$ in
an arbitrary way and let 
$$
U_\epsilon=\dbar(\rho_\epsilon v).
$$
By the Lemma this a $\dbar$-closed, or even exact, extension of $u$
with $L^2$-norm going to zero with $\epsilon$.

For the converse, assume there are some $\dbar$-closed
extensions, $U_\epsilon$,   with
$L^2$-norms going to 
zero. Let $\mathcal{U}_\epsilon$ be the harmonic representative of the
cohomology class $[U_\epsilon]$. The norms of the harmonic
representatives are smaller, so they go to zero too. Now, the space of
harmonic forms is finite dimensional, so all norms are
equivalent. Hence the supnorms of $\mathcal{U}_\epsilon$ also go to zero, so
the restrictions of $\mathcal{U}_\epsilon$ to $\Delta$ also go to
zero. Since on $\Delta $, 
$u-\mathcal{U}_\epsilon$ is exact, it follows that $u$ lies in the closure of
the space of exact forms. But $\dbar$ has closed range on a compact
manifold, so $u$ must be exact.

\section{ $\dbar$-closed  forms}

In this section we adapt the argument in \cite{Berndtsson} to forms of
higher degree. We will use the residue formulation of the extension
problem and the set up is as follows. 

$X$ is a compact K\"ahler manifold, with K\"ahler form $\omega$  and
$L$ is a holomorphic line 
bundle over $X$. $\Delta$ is a smooth divisor in
$X$, given as $\Delta=s^{-1}(0)$, with $s$ a holomorphic section of a
line bundle $S$. Let $u$ be a smooth $L$-valued $\dbar$-closed $(n-1,q)$-form on
$\Delta$. We want to find a smooth $L$-valued $\dbar$-closed
$(n,q)$-form, $U$, on $X$, 
such that
\be
U=ds\wedge u
\ee
on $\Delta$. Note that $u$ could alternately be interpreted as a
$(0,q)$-form on $\Delta$ with values in $K_\Delta +L$. By the
adjunction isomorphism
$$
u\mapsto ds\wedge u
$$
between $K_\Delta$ and $(K_X+ S)|_Y$ this means that we extend a $(0,q)$-form
with values in
\be
F:= K_X+S+L
\ee
to a form with values in $F$. 
\begin{thm}
Assume that $\phi$ is a smooth metric on $L$ and that $\psi$ is a
smooth metric on $S$ such that
$$
i\ddbar\phi\wedge \omega^q\geq \epsilon\, i\ddbar\psi\wedge \omega^q
$$
and
$$
i\ddbar\phi\wedge \omega^q\geq 0.
$$
Assume moreover the normalizing inequality
$$
\log |s|^2 e^{-\psi}\leq -1/\epsilon.
$$

 Let $u$ be a smooth $\dbar$-closed $(n-1,q)$-form with
  values in $L$ over $\Delta$. Then there is a $\dbar$-closed
  $(n,q)$-form, $U$, with values in $S+L$ over $X$ such that
$$
U=ds\wedge u
$$
on $\Delta$ and
$$
\int_X |U|^2 e^{-\phi-\psi}dV_X\leq C_0 \int_\Delta
|u|^2 e^{-\phi}dV_Y
$$
where $C_0$ is an absolute constant. The norms and the volume
forms are defined by the K\"ahler form $\omega$. 
\end{thm}

The arguments starts with the observation that if $U$ satisfies the
conclusion of the theorem, and if $v:=U/s$, then $v$ has values in
$K_X+L$ and solves
\be
\dbar v = \dbar(1/s)\wedge ds\wedge u= c u\wedge [\Delta]
\ee
where $[\Delta]$ is the current of integration on
$\Delta$. Conversely, let $v$ solve  (3.3) and assume that $U:=sv$ is
smooth. On $\Delta$ we can write $U=ds\wedge \tilde u$ by the
adjunction isomorphism. Then
$$
\dbar v = \dbar(1/s)\wedge ds\wedge \tilde u= c \tilde u\wedge [\Delta].
$$ 
Hence $\tilde u=u$ on $\Delta$, so $U$ solves the extension problem. 

We now try (and fail!) to solve this $\dbar$-problem and start to give
it a dual formulation. Let
$$
f:= u\wedge [\Delta],
$$
a current with measure coefficients, concentrated on $\Delta$ and of 
bidegree $(n, q+1)$. The proof of the next lemma will be postponed to
the end of the section. 
\begin{lma}{\bf ( The basic estimate )}
Assume, in addition to the assumptions in Theorem 3.1,  that 
$$
\|u\|^2_\Delta\leq 1
$$
Then, for any smooth $L$-valued $(n,q)$-form $\alpha$ on $X$
$$
|\langle f,\alpha\rangle|^2\leq C_0\int_X
\frac{|\dbar^*_\phi\alpha|^2}{|s|^2 e^{-\psi}}e^{-\phi}.
$$
\end{lma}
The norm $\|\cdot\|_\Delta$ here is the $L^2$-norm defined by the
K\"ahler form $\omega$ and the metric $\phi$ on $L$. 

\bigskip

\noindent
Now consider the conjugate linear functional
$$
R(\dbar^*_\phi\alpha)=\langle f,\alpha\rangle
$$ 
defined on the space 
$$
E:=\{\dbar^*_\phi\alpha;\, \alpha\,\, \text{smooth}\}.
$$

 By
the lemma, $R$ is bounded  
by the norm
$$
\int_X
\frac{|\dbar^*_\phi\alpha|^2}{|s|^2 e^{-\psi}}e^{-\phi}
$$
on the
subspace $E_0$ of elements of $E$ such that this norm is finite. Clearly, this
subspace consists of forms $\dbar^*_\phi\alpha$   that
vanish on $\Delta$. 
By the Riesz representation theorem, there is a form $w$ such that
$$
R(\dbar^*_\phi\alpha)=\int_X\frac{w\cdot\overline{\dbar^*_\phi\alpha}}{|s|^2
    e^{-\psi}}e^{-\phi} 
$$
for all $\alpha$ with  $\dbar^*_\phi\alpha=0$ on $\Delta$. Moreover, $w$
can be taken to satisfy
$$
\int_X
\frac{|w|^2}{|s|^2 e^{-\psi}}e^{-\phi}\leq C_0.
$$
Substitute
$$
v=w/(|s|^2 e^{-\psi}).
$$
Then 
\be
\langle f,\alpha\rangle=R(\dbar^*_\phi\alpha)=\int_X v\cdot
\overline{\dbar^*_\phi\alpha}e^{-\phi}
\ee
and
\be
\int_X |v|^2 |s|^2 e^{-\phi-\psi}\leq C_0.
\ee
Notice that this {\it does not mean} that  $\dbar v=f$ since we only know that
(3.4) holds for $\alpha$ with  $\dbar^*_\phi\alpha=0$ on $\Delta$.

In order to get smoothness we now choose $v$ with minimal norm defined
in (3.5), and the first objective is to check that there is a
minimizer.
\begin{lma}
Assume that $v_k$ is a sequence of forms such that
$$
\langle f,\alpha\rangle=\int_X v_k\cdot
\overline{\dbar^*_\phi\alpha}e^{-\phi}
$$
for all $\alpha$ with $\dbar^*_\phi\alpha=0$ on the divisor. Assume also that
$$
\int_X  |v-v_k|^2 |s|^2 e^{-\phi-\psi}\rightarrow 0
$$
for some $v$ satisfying
$$
\int_X |v|^2 |s|^2 e^{-\phi-\psi}<\infty.
$$
Then 
$$
\langle f,\alpha\rangle=\int_X v\cdot
\overline{\dbar^*_\phi\alpha}e^{-\phi}
$$
for all $\alpha$ with $\dbar^*_\phi\alpha=0$ on the divisor. 
\end{lma}
This means that the affine space of forms $v$ that satisfy (3.4) is
closed for the norm in (3.5), so it has an element of minimal
norm. The proof of the lemma is clear since
$$
|\int_X (v-v_k)\cdot\overline{\dbar^*_\phi\alpha} e^{-\phi-\psi}|^2\leq
\int_X  |v-v_k|^2 |s|^2 e^{-\phi-\psi} \int_X
\frac{|\dbar^*_\phi\alpha|^2}{|s|^2 e^{-\psi}}e^{-\phi} .
$$
The next point is to see that if $v$ is a minimizer, then $sv$ is a
harmonic form, hence smooth.
\begin{lma}
Assume that $v$ minimizes the norm in (3.5) among all solutions to
(3.4). Then $\dbar^*_{\phi+\psi}(sv)=0$.
\end{lma}
\begin{proof}
If $v$ is a minimizer then
$$
\int_X  |v|^2 |s|^2 e^{-\phi-\psi}\leq \int_X  |v-\dbar u|^2 |s|^2
e^{-\phi-\psi} 
$$
for all smooth $u$. This means that
$$
\int_X sv\cdot\overline{ \dbar su} e^{-\phi-\psi}=0.
$$
Hence $\dbar^*_{\phi+\psi} sv=0$ at least outside of $\Delta$. But
$sv$ has finite 
$L^2$-norm so a divisor is removable for this equation.( A
$\dbar^*$-equation for a form is a $\dbar$-equation for $*$ of the form.)
\end{proof}
Finally we have 
\begin{lma}
$$
\dbar(sv)=0.
$$
\end{lma}
\begin{proof}
Since $sv$ takes values in  $L+S$, we have to check that
$$
\langle sv, \dbar^*_{\phi+\psi}\xi\rangle_{\phi+\psi}=0
$$
for any smooth $(n,q)$-form $\xi$ with values in $L+S$. For this, note
first that
$$
\sigma:=\bar s e^{-\psi}= |s|^2 e^{-\psi}/s
$$
is a smooth section with values in $-S$, which vanishes on
$\Delta$. Therefore $\alpha_\xi:=\sigma \xi$ is $L$-valued and vanishes on
$\Delta$. Moreover, one easily checks that
\be
\bar s \dbar^*_{\phi+\psi}\xi=e^{\psi}\dbar^*_\phi\alpha_\xi.
\ee
Hence
$$
\langle sv, \dbar^*_{\phi+\psi}\xi\rangle_{\phi+\psi}=\langle
v,\dbar^*_\phi\alpha_\xi \rangle_\phi=\langle f,\alpha_\xi\rangle_\phi,
$$
where the last equality follows from (3.4). We are allowed to apply
(3.4) because  $\dbar^*_\phi\alpha_\xi$ is zero on $\Delta$ by
(3.6). Since $f$ is supported on 
$\Delta$ where $\alpha_\xi$ vanishes, $\langle
f,\alpha_\xi\rangle_\phi$  equals zero, and we are
done. 
\end{proof}

\bigskip

All in all we have now seen that $U:=sv$ is harmonic and therefore
smooth, if $v$ is the minimal solution of the dual problem. What
remains is to investigate the behaviour of $U$ on the divisor. Write
$U=ds\wedge \tilde u$ on the divisor. Let $\alpha$ be a smooth
$L$-valued $(n,q+1)$-form
such that $\dbar^*_\phi\alpha=0$ on the divisor and write
$$
\alpha = \gamma_\alpha\wedge \omega^{q+1}/(q+1)!
$$
for some (uniquely determined) $(n-q-1,0)$-form $\gamma_\alpha$. 
Then for any $(n,q+1)$-form $g$
$$
\langle g,\alpha\rangle_\omega \omega^n/n!=g\wedge\bar\gamma_\alpha,
$$
(see \cite{PCMI} for more on this).

\bigskip

\noindent
Hence
$$
\langle f,\alpha\rangle=\int_X f\wedge \bar\gamma_\alpha e^{-\phi}=
\int_\Delta u\wedge \bar\gamma_\alpha e^{-\phi}. 
$$
On the other hand, by (3.4) this also equals
$$
\int_X v\cdot
\overline{\dbar^*_\phi\alpha}e^{-\phi}=\int_X U/s\cdot
\overline{\dbar^*_\phi\alpha}e^{-\phi}=\int_X \dbar(1/s)\wedge U\wedge
\bar\gamma_\alpha e^{-\phi}=\int_\Delta\tilde u\wedge
\bar\gamma_\alpha e^{-\phi}.
$$
From this we see that
\be
\int_\Delta u\wedge \bar\gamma_\alpha e^{-\phi}=\int_\Delta \tilde u\wedge
\bar\gamma_\alpha e^{-\phi}
\ee
for all $\alpha$ such that $\dbar^*_\phi\alpha=0$ on $\Delta$.  This latter
condition is equivalent to saying that
$$
\dbar(\bar\gamma_\alpha e^{-\phi})=0.
$$
Let $\bar\gamma_\alpha e^{-\phi}=:\chi$. This is a $(0, n-q-1)$-form with
values in $-L$. 
Hence
\be
\int_\Delta (u-\tilde u)\wedge\chi=0
\ee
for all $(0, n-q-1)$-forms  $\chi$ with values in $-L$ such that
$\dbar\chi=0$ on $\Delta$. The $\dbar$ 
operator here is the $\dbar$ on $X$, but, by the next lemma, the same
thing holds if only $\dbar_\Delta\chi=0$.
\begin{lma} Let $\chi$ be a smooth $-L$-valued $(0,p)$-form on $X$ such
  that $\dbar_\Delta\chi=0$ on $\Delta$. Then there is a smooth form
  on 
  $X$, $\tilde\chi$ 
  such that $\dbar_X\tilde\chi=0$ on $\Delta$ and $\chi=\tilde\chi$ on
  $\Delta$.
\end{lma}
\begin{proof} Locally the divisor is given by an equation $z_1=0$ in
  some local chart. The hypothesis then means that $\dbar\chi$ is
  divisible by $d\bar z_1$. To get a local extension it therefore
  suffices to subtract a suitable multiple of $\bar z_1$, and one then
  obtains $\tilde\chi$ from a partition of unity. 
\end{proof}
It follows from the lemma that (3.8) holds for any $\chi$ on $\Delta$
such that $\dbar_\Delta\chi=0$. 
 But this means that  the
difference $u-\tilde u$ is $\dbar$-exact. Hence we have proved
Koziarz's theorem that $u$ can be extended up to an exact error, and
the proof of Theorem 3.1 then follows from Proposition 2.1. 

\bigskip

All that remains is now to prove Lemma 3.2.

\subsection{Proof of the basic estimate}
This follows closely the
proof in \cite{Berndtsson}, and the proof in the compact case is
described in \cite{PCMI}, and we refer to these notes for more details
on the computations that follow. 

We first write as above 
$$
\alpha=\gamma\wedge\omega^{q+1}/(q+1)!
$$
so that $\gamma$ is an $L$-valued $(n-q-1,0)$-form. 
Then define a scalar valued $(n-1,n-1)$-form
$$
T_\alpha= c_q\gamma\wedge \bar\gamma\wedge\omega^q e^{-\phi}/q!
$$
where $c_q$ is a unimodular constant chosen so that $T_\alpha$ is a
positive form. We will prove the basic estimate first assuming that
$\dbar\alpha=0$. In that case it follows from  Proposition 3.4.1 in
\cite{PCMI} that 
\be
i\ddbar T_\alpha\geq -2\Re (\dbar\dbar^*_\phi\alpha,\alpha)\omega^n/n!
+ i\ddbar\phi\wedge T_\alpha.
\ee
Let
$$
W:=-\log( |s|^2 e^{-\psi}).
$$
By the hypothesis in Theorem 3.1, $W\geq 1/\epsilon$. Moreover
$$
i\ddbar W=i\ddbar\psi -c[\Delta].
$$
Multiply (3.9) by $W$ and apply Stokes' formula. This gives
\be
\int_X
(Wi\ddbar\phi\wedge\omega^q-i\ddbar\psi\wedge\omega^q)/q!\wedge c_q \gamma\wedge
\bar\gamma e^{-\phi} +c\int_\Delta c_q\gamma\wedge\bar\gamma\wedge \omega^q/q!
e^{-\phi}\leq
2\Re\langle\dbar\dbar^*_\phi\alpha, W\alpha\rangle.
\ee
By the hypotheses in Theorem (3.1) the first integral in the left hand
side is nonnegative, so we get
$$
c\int_\Delta c_q\gamma\wedge\bar\gamma\wedge\omega^q/q!
e^{-\phi}\leq
2\Re\langle\dbar\dbar^*_\phi\alpha, W\alpha\rangle.
$$
On the other hand
$$
|\langle f,\alpha\rangle|^2=|\int_X f\wedge \bar\gamma e^{-\phi}|^2=
|\int_\Delta u\wedge\bar\gamma e^{-\phi}|^2.
$$
By the Cauchy inequality we get, since by assumption $\|u\|_\Delta\leq
1$ that
$$
|\langle f,\alpha\rangle|^2\leq \|\gamma\|^2_\Delta=\int_\Delta
c_q\gamma\wedge\bar\gamma\wedge\omega^q/q! e^{-\phi}\leq
2c^{-1}\Re\langle\dbar\dbar^*_\phi\alpha, W\alpha\rangle.
$$
The right hand side equals
$$
2\int_X W|\dbar^*_\phi\alpha|^2 e^{-\phi}-2\Re\langle \dbar
W\wedge\dbar^*_\phi\alpha, \alpha\rangle.
$$
The first term is obviously OK since $W\leq e^W$. For the second term
we write
$$
II:=\langle \dbar W\wedge\dbar^*_\phi\alpha, \alpha\rangle=
\int_X \dbar W\wedge\dbar^*_\phi\alpha\wedge\bar\gamma_\alpha
e^{-\phi}.
$$
By Cauchy's inequality
$$
2|II|\leq
\int_X\frac{|\dbar^*_\phi\alpha|^2}{|s|^2e^{-\psi}}e^{-\phi}+
c_q\int_X e^{-W}\partial W\wedge\dbar
W\wedge\gamma_\alpha\wedge\bar\gamma_\alpha\wedge\omega^q/q!e^{-\phi}.
$$
It is only the last term that we need to worry about. Let
$$
W_1=(1-e^{-W}).
$$
Then $0<W_1<1$ and 
$$
i\ddbar W_1=-e^{-W}i\partial W\wedge\dbar W.
$$
We now repeat the same argument as above, but with $W$ replaced by
$W_1$. The result is
$$
c_q\int_X e^{-W}\partial W\wedge\dbar
W\wedge\gamma_\alpha\wedge\bar\gamma_\alpha\wedge\omega^q/q!e^{-\phi}\leq
2\int_X W_1|\dbar^*_\phi\alpha|^2 e^{-\phi}-2\Re\langle \dbar
W_1\wedge\dbar^*_\phi\alpha, \alpha\rangle.
$$
The first term is controlled since $W_1<1$ and the second term can
easily be absorbed in the left hand side. This completes the proof of
the basic estimate in case $\dbar\alpha=0$.

The general case is easily reduced to this special case. We decompose
$$
\alpha=\alpha^1+\alpha^2
$$
where $\alpha^1$ is $\dbar$-closed and $\alpha^2$ is orthogonal to the
space of $\dbar$-closed forms. Then in particular $\alpha^2$ is
orthogonal to $\dbar$-exact forms, so $\dbar^*_\phi\alpha^2=0$. Hence
$\alpha^1$ satsifies $\dbar\alpha^1=0$ and
$\dbar^*_\phi\alpha^1=\dbar^*_\phi\alpha$. This means, by elliptic
regularity that $\alpha^1$, and therefore $\alpha^2$ are both
smooth. Now we claim that booth sides in the basic estimate are
unchanged if we replace $\alpha$ by $\alpha^1$. Since we know the
basic estimate holds for $\alpha^1$ this is all we need. That the
right hand side is unchanged we have already seen. That the left hand
side is unchanged follows since $f$ is closed and $\alpha^2$ is
orthogonal to closed forms. There is a minor problem here, coming from
the fact that $f$ is not an $L^2$-form. However, $f$ is cohomologous
to a smooth form
$$
f=f_{\text{smooth}}+\dbar g
$$
and this proves the claim since $\alpha^2$ is smooth and satsifies
$\dbar^*_\phi\alpha^2=0$.

\def\listing#1#2#3{{\sc #1}:\ {\it #2}, \ #3.}


\begin{thebibliography}{9999}

\bibitem{Amar}\listing{Amar, E}{Extension de fonctions holomorphes et courants}{ Bull. Sciences Mathématiques, 2ème série, 107 (1983), 25-48}
 

\bibitem{Berndtsson}\listing{Berndtsson, B}{The extension theorem of
    Ohsawa-Takegoshi and the theorem of Donnelly-Fefferman.} {
    Ann. Inst. Fourier (Grenoble)  46  (1996),  no. 4, 1083 -1094 }
   

\bibitem{PCMI}\listing{Berndtsson, B}{An Introduction to things $\dbar$
   }{ Lecture Notes from the PCMI } 

\bibitem{Demailly}\listing{Demailly, J-P}{On the
    Ohsawa-Takegoshi-Manivel L 2 extension theorem}{ Complex analysis
    and geometry (Paris, 1997). Prog. Math., vol. 188, pp. 47 -82}
\bibitem{Koziarz}\listing{Koziarz, V}{Extensions with estimates of cohomology classes.
   }{  Manuscripta Math.  134  (2011),  no. 1-2, 43-58 } 

\bibitem{Manivel}\listing{Manivel, L}{Un théorème de prolongement $L\sp 2$L2  de sections holomorphes d'un fibré hermitien
   }{  Math. Z.  212  (1993),  no. 1, 107 -122 } 

\bibitem{Ohsawa}\listing{Ohsawa, T and Takegsohi, K}{On the extension of $L^2$L2  holomorphic functions.
   }{  Math. Z.  195  (1987),  no. 2, 197 - 204.  } 

\bibitem{Siu}\listing{Siu, Y-T}{A vanishing theorem for semipositive
    line bundles over non-Kähler manifolds.}{ J. Diff. Geom. 19, 431 -
    452}
\end{thebibliography}
\end{document}